\documentclass[12pt,a4paper]{article}
\usepackage{amsmath}
\usepackage{amssymb}
\usepackage{amsthm}

\usepackage{verbatim}

\usepackage[utf8]{inputenc}

\usepackage{graphicx}
\usepackage{color}
\usepackage{enumerate}

\usepackage{esint}
\usepackage[all]{xy}
\usepackage{framed}
\usepackage{bbm}
\usepackage{subcaption}

\theoremstyle{plain}
\newtheorem{theorem}{Theorem}[section]
\newtheorem{corollary}[theorem]{Corollary}
\newtheorem{lemma}[theorem]{Lemma}

\newtheorem{proposition}[theorem]{Proposition}

\theoremstyle{definition}
\newtheorem{definition}{Definition}[section]
\newtheorem{remark}{Remark}[section]
\newtheorem{example}{Example}[section]
\numberwithin{equation}{section}

\newcommand{\mR}{\mathbb{R}}

\newcommand{\bfx}{\mathbf{x}}

\newcommand{\bfc}{\mathbf{c}}

\newcommand{\bff}{\mathbf{f}}

\newcommand{\bfq}{\mathbf{q}}

\newcommand{\bfe}{\mathbf{e}}
\newcommand{\Hu}{\mathcal{H}u}

\newcommand{\M}{\mathcal{M}}

\newcommand{\dd}{\mathrm{\,d}}
\newcommand{\DD}{\mathrm{D}}

\newcommand{\oll}{\overline{\lambda}}

\DeclareMathOperator{\tr}{tr}
\DeclareMathOperator{\diag}{diag}

\DeclareMathOperator{\spa}{span}

\begin{document}

\author{Karl K. Brustad\\ {\small Aalto University}}
\title{Total derivatives of eigenvalues and eigenprojections of symmetric matrices}

\maketitle

\begin{abstract}
Conditions for existence and formulas for the first- and second order total derivatives of the eigenvalues, and the first order total derivatives of the eigenprojections of smooth matrix-valued functions $H\colon\Omega\to S(m)$ are given. The eigenvalues and eigenprojections are considered as functions in the same domain $\Omega\subseteq\mR^n$.
\end{abstract}

\section{Introduction}
Let $\Omega\subseteq\mathbb{R}^n$ be a domain and assume that $H\colon\Omega\to S(m)$ is a continuously differentiable function taking values in the space $S(m)$ of symmetric $m\times m$ matrices.
Under what conditions are the eigenvalues and the eigenprojections of $H$ differentiable, and what are their total derivatives? The eigenprojection $P_j(x)$, corresponding to the eigenvalue $\lambda_j(x)$ of $H$ at $x\in\Omega$, is the unique symmetric $m\times m$ projection matrix, i.e. $P_j^T(x) = P_j(x) = P_j^2(x)$, satisfying
\[H(x)P_j(x) = \lambda_j(x)P_j(x)\]
with rank, or \emph{dimension}, equal to the multiplicity of the eigenvalue.

The Hessian matrix $\Hu$ of a function $u\in C^3(\Omega)$ is a motivating special case.
Then $m=n$ and, as a standard example -- showing that smooth matrices need not have differentiable eigenvalues -- one may consider the real part of the analytic function $z^3$ in the plane $\mathbb{C}$: If we set $u(x,y) := \frac{1}{6}(x^3 - 3xy^2)$, then
\[\Hu(x,y) = \begin{pmatrix}
x & -y\\ -y & -x
\end{pmatrix}\]
with eigenvalues
\[\lambda_1(x,y) = -\sqrt{x^2+y^2}\qquad\text{and}\qquad \lambda_2(x,y) = \sqrt{x^2+y^2}.\]
We see that the problem occurs at the origin where the eigenvalues "cross". This is a well known phenomenon.
The corresponding eigenprojections are not even continuous since $\tr P_1 = \tr P_2 = 1$ away from the origin while $P_1(0,0) = P_2(0,0) = I$ as $\Hu(0,0) = 0 = 0\cdot I$.

Perturbations of eigenvalues and eigenvectors of symmetric matrix-valued functions have been studied in various settings. It is shown in \cite{MR1868352} that the $j$'th eigenvalue of $H(t) =  H_0 + tH_1 + \tfrac{1}{2}t^2 H_2$ always has first- and second order one-sided derivatives. This work is partly based on \cite{MR1320701} and is developed further in \cite{MR3089437}. Our formulas for the derivatives of $\lambda_j$ have counterparts in these papers, although the setting is not exactly the same.
The total projection for the \emph{$\lambda$-group} -- i.e. the sum of projections corresponding to neighbouring eigenvalues -- is analyzed in \cite{MR1335452}. The expression \eqref{eq:eigenproj_dir_thm} for the derivative of $P_j$ may be compared with the one found in (Theorem 5.4 \cite{MR1335452}). The book by Kato is a standard reference for perturbating matrices depending on a single real or complex parameter. Unfortunately, many of the results therein do not generalize if the matrix depends on several variables.
We also mention the papers \cite{MR1871318} and \cite{MR1238911} where, respectively, \emph{spectral functions} and solutions to nonlinear eigenvalue-eigenvector problems are differentiated.

We shall consider the eigenvalues and eigenprojections as functions in $\Omega\subseteq\mR^n$.
The total derivative of an eigenvalue $\lambda_j$ is, if it exists, the gradient $\nabla\lambda_j$.
For the matrix-valued eigenprojections $P_j$, the total derivative is a mapping $\mathbb{R}^m\times\mathbb{R}^m\times\mathbb{R}^n\times\Omega\to\mathbb{R}$ linear in the three first arguments. That is, a third order tensor. In order to simplify the notation and minimize the use of indexes, we introduce two different \emph{first order matrix valued} tensors representing the derivative of matrix functions. The (double-sided) directional derivatives are also studied. In contrast to the one-sided limits, they do not always exist as our example clearly shows.

Our main results Theorem \ref{thm:projectionderivative}, Theorem \ref{thm:projection_directional_derivative}, and Theorem \ref{thm:lambda_reg} give
explicit expressions for the derivatives of $\lambda_j$ and $P_j$ in terms of $H$ and its derivatives.
As a little surprise, it turns out that an eigenprojection is continuous only if it is differentiable, and it has constant dimension only if it is directionally differentiable.  Moreover, if $H$ is $C^2$ and the eigenprojection is continuous, then the corresponding eigenvalue has a differentiable gradient.
We have not been able to find these observations in the litterature.
A key ingredient in proving the theorems is Lemma \ref{lem:derofmin} (Lemma 5.2 \cite{MR3155251}). It does not seem to have been used in the other aforementioned papers.

\section{Preliminaries}\label{sec:prelim}

The matrix norm used throughout the paper is $\lVert X\rVert := \sqrt{\tr(X^TX)}$.
Even though we treat $\mR^m$ as $\mR^{m\times 1}$ algebraically, the vector norm is denoted by $\lvert y\rvert :=\sqrt{y^Ty}$.  If $\bff\colon \Omega\to\mR^m$ is a differentiable function, its \emph{Jacobian matrix} is the mapping $\nabla\bff\colon\Omega\to\mR^{m\times n}$ satisfying
\[\bff(x+y) = \bff(x) + \nabla\bff(x)y + o(\lvert y\rvert)\]
as $y\to 0$. In particular, gradients are row vectors.

\subsection{Matrix derivatives}

\begin{definition}
Let $F\colon \Omega\to\mathbb{R}^{m\times k}$ be given. The \emph{directional derivative} $\DD F\colon\mR^n\times\Omega\to\mR^{m\times k}$ of $F$ is defined by
\begin{equation}\label{def2}
\DD_e F(x) := \lim_{h\to 0}\frac{F(x+he) - F(x)}{h}
\end{equation}
whenever the limit exists.
When $F$ is differentiable, the \emph{Jacobian derivative} $\nabla F\colon\mR^k\times\Omega\to\mR^{m\times n}$ of $F$ is defined by
\begin{equation}\label{def1}
\nabla_q F(x) := \nabla[Fq](x).
\end{equation}
That is, the Jacobian matrix of the vector valued function $x\mapsto F(x)q$.
\end{definition}

It is possible to define the Jacobian in terms of combinations of partial derivatives, but we shall reserve the notation $\nabla$ and $\nabla_q$ for functions that are assumed to be differentiable.

Clearly, $\DD_e F^T = (\DD_e F)^T$ and any symmetry of a square matrix $F$ is therefore preserved.
If $F$ is assumed to be differentiable, the directional derivative satisfies
\[F(x+y) = F(x) + \DD_y F(x) + o(\lvert y\rvert)\qquad \text{as $y\to 0$,}\]
and one can check that
\begin{equation}\label{sym1}
\DD_e F(x)q = \nabla_q F(x)e\qquad \forall e\in\mathbb{R}^n,\, q\in\mathbb{R}^k.
\end{equation}
Note that the dimensions match and that the above is an equality in $\mathbb{R}^m$.

If $\bfq\colon\Omega\to\mathbb{R}^k$ and $\bfe\colon [a,b]\to\mR^n$ are functions, we write
\[\nabla_{\bfq(x)} F(x) := \nabla_q F(x)\Big\vert_{q = \bfq(x)}\qquad\text{and}\qquad \DD_{\bfe(t)} F(x) := \DD_e F(x)\Big\vert_{e = \bfe(t)}.\]
Thus if $\bfq$ is differentiable, the product rule yields
\[\nabla[F\bfq](x) = F(x)\nabla\bfq(x) + \nabla_{\bfq(x)} F(x),\]
and if $\bfc\colon [a,b]\to\Omega$ is a differentiable curve, we get, by the chain rule and by using \eqref{sym1}, that
\begin{equation}\label{chainrule}
\frac{\dd}{\dd t}F(\bfc(t)) = \DD_{\bfc'(t)}F(\bfc(t)).
\end{equation}
Moreover, for vectors $p\in\mR^m$ and $q\in\mR^k$ we have
\begin{equation}\label{sym2}
p^T\nabla_q F(x) = q^T\nabla_p F^T(x).
\end{equation}
Note again that the dimensions match and that \eqref{sym2} is an equality in $\mR^{1\times n}$. Indeed, since $F^T(x)$ is a $k\times m$ matrix, the Jacobian $\nabla_p F^T = \nabla[F^Tp]$ is of dimension $k\times n$.

If the matrix function $F$ is a Hessian $\Hu\colon \Omega\to S(n)$ of $u\in C^3(\Omega)$, then the directional and the Jacobian derivatives coincide and are again symmetric, that is,
\begin{equation}\label{hessprop}
\DD_\xi\Hu(x) = \nabla_\xi\Hu(x)\in S(n)\qquad\forall\xi\in\mR^n.
\end{equation}
It is, in fact, the Hessian matrix of the $C^2$ function $x\mapsto \nabla u(x)\xi$ in $\Omega$. Combined with \eqref{sym2}, this means that
\begin{equation}\label{eq:Hess_sym}
\xi_i^T\DD_{\xi_j}\Hu\,\xi_k = \xi_{\pi(i)}^T\DD_{\xi_{\pi(j)}}\Hu\,\xi_{\pi(k)}
\end{equation}
in $\Omega$ for all $\xi_i,\xi_j,\xi_k\in\mR^n$ and all permutations $\pi$ on $\{i,j,k\}$.
In particular, $e_i^T\DD_{e_j}\Hu\,e_k = \frac{\partial^3 u}{\partial x_i \partial x_j \partial x_k}$.

\subsection{Symmetric matrices}

The \emph{spectral theorem} states that every symmetric $m\times m$ matrix can be diagonalized.
For any $X\in S(m)$ there exists an orthogonal $m\times m$ matrix $U$ such that $U^TXU = \diag(\lambda_1,\dots,\lambda_m)$ where $\lambda_1\leq\cdots\leq\lambda_m$ are the eigenvalues of $X$.
Moreover, the \emph{eigenspaces} $E_j := \{\xi\in\mR^m\,\vert\, X\xi = \lambda_j\xi\}$ are $d_j$-dimensional subspaces of $\mR^m$ where $d_j$ is the multiplicity of $\lambda_j$.
The spaces $E_j$ and $E_k$ are orthogonal whenever $\lambda_j\neq \lambda_k$. Obviously, $E_j=E_k$ if $\lambda_j= \lambda_k$.
By writing $U = (\xi_1,\dots,\xi_m)$, we get that
\begin{equation}\label{eq:Xdiag}
X = U\diag(\lambda_1,\dots,\lambda_m)U^T = \sum_{i=1}^m \lambda_i\xi_i\xi_i^T
\end{equation}
and that $E_j = \spa\{\xi_i\,\vert\, \lambda_i = \lambda_j\}$.

The class of symmetric $m\times m$ projection matrices is denoted by
\[Pr(m) := \{P\in S(m)\,\vert\, PP = P\}.\]
Since their eigenvalues are either 0 or 1, these matrices are on the form
\begin{equation}\label{eq:projsplit}
P = \sum_{i=1}^d \xi_i\xi_i^T = QQ^T,\qquad Q := (\xi_1,\dots,\xi_d)\in\mR^{m\times d},
\end{equation}
for some $d = 0,1,\dots, m$ (with the convention that empty sums are zero) and where $Q^TQ = I_d$. The set $\{\xi_1,\dots,\xi_d\}$ is an orthonormal basis for the $d$-dimensional subspace
\[P(\mR^m) := \{P\xi\,\vert\, \xi\in\mR^m\}\subseteq \mR^m.\]
Conversely, given a subspace $E$ of $\mR^m$, there is a unique symmetric projection $P$ such that $E = P(\mR^m)$.
Indeed, if $P(\mR^m) = E = R(\mR^m)$, then $P\xi,R\xi\in E$ for every $\xi\in\mR^m$. Thus $RP\xi = P\xi$ and $PR\xi = R\xi$ and $P = P^T = (RP)^T = PR = R$.
Note therefore that the factorization \eqref{eq:projsplit} is not unique as $P = \sum_{i=1}^d \eta_i\eta_i^T$ for every orthonormal basis $\{\eta_1,\dots,\eta_d\}$ of $P(\mR^m)$.

In the case of the symmetric matrix $X$ it follows that
\[P_j = \sum_{\substack{ i=1\\ \lambda_i = \lambda_j} }^m\xi_i\xi_i^T\]
is the unique eigenprojection corresponding to the $j$'th eigenvalue of $X$, regardless of the choice $U = (\xi_1,\dots,\xi_m)$ of eigenvectors.

If we let $\alpha\colon\{1,\dots,s\}\to \{1,\dots,m\}$ be a re-indexing that picks out all of the $s := \vert\{\lambda_1,\dots,\lambda_m\}\vert$ \emph{distinct} eigenvalues of $X$, we may collect the terms in \eqref{eq:Xdiag} with equal coefficients and write
\begin{equation}\label{eq:X_unrep}
X = \sum_{l=1}^s \lambda_{\alpha(l)}P_{\alpha(l)}.
\end{equation}
Now,
\[P_{\alpha(l)}P_{\alpha(k)} = \delta_{lk}P_{\alpha(l)}\qquad \text{and}\qquad \sum_{l=1}^sP_{\alpha(l)} = \sum_{i=1}^m\xi_i\xi_i^T = I\]
and \eqref{eq:X_unrep} is the unique representation of $X$ in terms of a complete set of eigenprojections and the \emph{unrepeated} eigenvalues.
However, since $\alpha$ depends on $X$ it is often more convenient to represent $X$ in terms of the \emph{repeated} versions $\lambda_i$ and $P_i$ where the indexing goes from 1 to $m$. This is obtained by noticing that
\[\sum_{\substack{ i=1\\ \lambda_i = \lambda_{\alpha(l)}} }^m\frac{\lambda_i}{d_i}P_i = \frac{\lambda_{\alpha(l)}}{d_{\alpha(l)}}P_{\alpha(l)}\sum_{\substack{ i=1\\ \lambda_i = \lambda_{\alpha(l)}} }^m 1= \lambda_{\alpha(l)}P_{\alpha(l)}\]
and thus
\[X = \sum_{l=1}^s \lambda_{\alpha(l)}P_{\alpha(l)} = \sum_{l=1}^s\sum_{\substack{ i=1\\ \lambda_i = \lambda_{\alpha(l)}} }^m\frac{\lambda_i}{d_i}P_i = \sum_{i=1}^m\frac{\lambda_i}{d_i}P_i.\]

In \cite{MR1091716}, the unrepeated eigenprojections are called the \emph{Frobenius covariants} and an explicit formula in terms of $X$ and the eigenvalues is given. In our notation
\begin{equation}\label{eq:frobenius}
P_{\alpha(k)} = \prod_{\substack{l=1\\ l\neq k}}^{s}\frac{X - \lambda_{\alpha(l)}I}{\lambda_{\alpha(k)} - \lambda_{\alpha(l)}}
\end{equation}
with the convention that an empty product is the identity. The formula can also be verified directly from \eqref{eq:X_unrep}.

\section{Differentiation of the eigenprojections}

By the above discussion, we can write $H(x)$ as
\[H(x) = \sum_{l=1}^{s} \lambda_{\alpha(l)}(x)P_{\alpha(l)}(x) = \sum_{i=1}^m \frac{\lambda_i(x)}{d_i(x)}P_i(x)\]
where either representation is unique in its specific sense.
Here, $d_i(x) := \tr P_i(x)$ and \(s = s(x)\in \{1,\dots,m\}\) is the number of \emph{different} eigenvalues of \(H\) at \(x\).

For \(1\leq j\leq m\), let \(A_j\colon\Omega\to S(m)\) be given by
\[A_j(x) := \sum_{\substack{l=1\\ \lambda_{\alpha(l)}(x)\neq\lambda_j(x)}}^{s}\frac{P_{\alpha(l)}(x)}{\lambda_j(x) - \lambda_{\alpha(l)}(x)} = \sum_{ \substack{ i=1\\ \lambda_i(x) \neq \lambda_j(x)} }^m \frac{P_i(x)/d_i(x)}{\lambda_j(x) - \lambda_i(x)}.\]
It commutes with $H$, $A_jP_j = P_j A_j = 0$ and satisfies
\begin{align*}
A_j(\lambda_j I - H)
	&= \sum_{\substack{l=1\\ \lambda_{\alpha(l)}\neq\lambda_j}}^{s}\frac{P_{\alpha(l)}}{\lambda_j - \lambda_{\alpha(l)}} \sum_{k=1}^{s}(\lambda_j - \lambda_{\alpha(k)})P_{\alpha(k)}\\
	&= \sum_{\substack{l=1\\ \lambda_{\alpha(l)}\neq\lambda_j}}^{s}P_{\alpha(l)}\\
	&= I - P_j.
\end{align*}
It is therefore the \emph{pseudoinverse} of the singular matrix $\lambda_j I - H$.

\begin{theorem}[Total derivative of eigenprojections]\label{thm:projectionderivative}
Let \(H\in C^1(\Omega,S(m))\) with repeated eigenvalues and eigenprojections $\lambda_i(x)$ and $P_i(x)$, $i=1,\dots,m$. Let $j\in\{1,2\dots,m\}$ and assume that either
\begin{equation}\label{eq:projectionderivative_condition}
\text{$P_j$ is continuous}\quad\textbf{or}\quad \text{$\lambda_j$ is differentiable and $\tr P_j$ is constant}
\end{equation}
in $\Omega$.
Then \(P_j\) is differentiable in \(\Omega\) with total derivatives
\[\nabla_q P_j(x) = P_j(x)\nabla_{A_j(x)q}H(x) + A_j(x)\nabla_{P_j(x)q}H(x)\]
and
\[\DD_e P_j(x) = P_j(x)\DD_e H(x)A_j(x) + A_j(x)\DD_e H(x)P_j(x)\]
for all $q\in\mR^m$ and all $e\in\mR^n$.
\end{theorem}

An immediate observation is that
\begin{equation}\label{eq:nablaPH}
\tr\left(\DD_e P_j H\right) = 0 = \tr\left(\DD_e P_j\right).
\end{equation}

\begin{proof}
We shall prove the claim under the latter assumption in \eqref{eq:projectionderivative_condition}. The proof of the theorem is then completed by Proposition \ref{prop:equivalentprop} which says that these two conditions are equivalent for $C^1$ matrices.
It is worth mentioning that the Frobenius formula \eqref{eq:frobenius} is not directly applicable since the indexing will depend on $x\in\Omega$.

We dropp the subscripts and write $P := P_j$, $d := \tr P$, $A := A_j$, and $\lambda := \lambda_j$. 
Fix \(x\in\Omega\), which we may assume to be the origin, and let $y\in\mR^n$ be small. By the differentiability assumptions 
\begin{align*}
0 &= \big( H(y) - \lambda(y)I\big)P(y)\\
&= \big( H - \lambda I + \DD_y H - \nabla\lambda y\cdot I\big)P(y) + o(|y|)
\end{align*}
as \(y\to 0\). Functions written without an argument are to be understood as evaluated at \(x=0\).
Multiplying from the left with \(A\) gives
\begin{equation}
(I-P)P(y) = A\left(\DD_y H - \nabla\lambda y\cdot I\right)P(y) + o(|y|) = O(|y|).
\label{eq:bigO}
\end{equation}
Since $A = A(I-P)$, it follows from \eqref{eq:bigO} that also
\begin{equation}\label{eq:(I-P)P_est}
\begin{aligned}
(I-P)P(y) &= A\DD_y H P(y) - \nabla\lambda y\cdot A(I-P)P(y) + o(|y|)\\
&= A\DD_y H P(y)+ o(|y|).
\end{aligned}
\end{equation}

It remains to find an estimate for $PP(y)$.
According to \eqref{eq:projsplit}, we can for each $y$ split $P(y)$ into a product $Q(y)Q^T(y)$. Although $Q(y)\in\mR^{m\times d(y)}$ is not unique, it is obviously bounded. Define
\[R(y) := Q^TQ(y) \in\mR^{d\times d(y)}\]
and write
\begin{equation}
\begin{aligned}\label{eq:RTR_est}
I_{d(y)} &= Q^T(y)Q(y)\\
&= Q^T(y)PQ(y) + Q^T(y)(I-P)Q(y)\\
&= R^T(y)R(y) + O(\vert y\vert^2)
\end{aligned}
\end{equation}
where the estimate on the last line is due to \eqref{eq:bigO} after multiplying on the right by $Q(y)$, and by the fact that $I-P = (I-P)^2$.
This means that the eigenvalues of \(R^T(y)R(y)\in S(d(y))\) are all in the range \(1 + O(|y|^2)\). It is therefore invertible, and the inverse is bounded as the eigenvalues of $(R^T(y)R(y))^{-1}$ are again in the range $(1 + O(|y|^2))^{-1} = 1 + O(|y|^2)$. Hence
\begin{equation}\label{eq:RTRinv_est}
(R^T(y)R(y))^{-1} = I_{d(y)} + O(\vert y\vert^2).
\end{equation}

We now use the assumption that $P(y)$ has constant dimension $d(y)\equiv d$. It implies that $R(y)$ is square and, by taking the determinant of \eqref{eq:RTR_est}, we see that $R(y)$ is invertible as well. The left-hand side of \eqref{eq:RTRinv_est} may therefore be written as
$R^{-1}(y) (R^T(y))^{-1}$, and multiplication from the left with $QR(y)$ and from the right with $R^T(y)Q^T$ then yields
\begin{equation}\label{eq:QTPQ_est}
P = QR(y)R^T(y)Q^T + O(\vert y\vert^2) = PP(y)P + O(\vert y\vert^2).
\end{equation}
Combining this with the transposed of \eqref{eq:(I-P)P_est} gives
\[PP(y) = PP(y)(P + I-P)\\
= P + PP(y)\DD_y H A + o(\vert y\vert)\]
and thus
\begin{align*}
P(y) &= PP(y) + (I-P)P(y)\\
     &= P + PP(y)\DD_y H A + A\DD_y H P(y)+ o(|y|).
\end{align*}
Since the whole expression is $P + O(|y|)$, the factor $P(y)\DD_y H$ can be replaced with $(P + O(|y|))\DD_y H = P\DD_y H + o(|y|)$ and we finally conclude that
\[P(y) = P + P\DD_y H A + A\DD_y H P+ o(|y|).\]

In order to get the formula for the directional derivative $\DD_e P$, substitute $y$ with $he$ where $e\in\mR^n$ and let the number $h$ go to zero. As for the Jacobian derivative, the symmetry \eqref{eq:Hess_sym} implies that
\begin{align*}
P(y)q - P(0)q &= \left(P\DD_y H A + A\DD_y H P \right)q+ o(|y|)\\
           &= \left(P\nabla_{Aq} H + A\nabla_{Pq} H \right)y+ o(|y|)
\end{align*}
and $\nabla_q P = P\nabla_{Aq} H + A\nabla_{Pq} H$ being the Jacobian matrix of $x\mapsto P(x)q$ at $x=0$.
\end{proof}

When inspecting the above proof, it becomes clear that the same formula for the directional derivative $\DD_e P_j$ would have been produced if $y$ had been replaced with $he$ all the way from the begining. But then the arguments work also for the weaker assumption of constant dimension and \textit{directional} differentibility of the eigenvalue. Since Proposition \ref{prop:diffoncurves} shows that the former of these two conditions implies the latter, the theorem below follows.

\begin{theorem}[Directional derivative of eigenprojections]\label{thm:projection_directional_derivative}
Let \(H\in C^1(\Omega,S(m))\) with repeated eigenvalues and eigenprojections $\lambda_i(x)$ and $P_i(x)$, $i=1,\dots,m$. Let $j\in\{1,2\dots,m\}$ and assume that $\tr P_j$ is constant in $\Omega$. Then \(P_j\) is directionally differentiable in \(\Omega\) with 
\begin{equation}\label{eq:eigenproj_dir_thm}
\DD_e P_j(x) = P_j(x)\DD_e H(x)A_j(x) + A_j(x)\DD_e H(x)P_j(x)
\end{equation}
for all $e\in\mR^n$.
\end{theorem}

The following counterexample settles the question whether differentiability of eigenvalues implies constant dimension of the eigenprojections.

\begin{example}[A Hessian matrix with crossing differentiable eigenvalues]
Let $u\colon\mR^2\to\mR$ be the real part of $\frac{1}{12}(x+iy)^4$. That is,
\[u(x,y) = \frac{x^4 - 6x^2y^2 + y^4}{12}.\]
Then
\[\Hu(x,y) = 
\begin{pmatrix}
x^2 - y^2 & - 2xy\\ -2xy & y^2 - x^2
\end{pmatrix}\]
with eigenvalues
\begin{align*}
\lambda_\pm(x,y) &= \frac{-\tr\Hu(x,y)}{2} \pm\frac{1}{2}\sqrt{\tr^2\Hu(x,y) - 4\det\Hu(x,y)}\\
            &= \pm\sqrt{ -\det\Hu(x,y)}\\
            &= \pm (x^2+y^2)
\end{align*}
that meet at the origin while still being differentiable.

One may check that $(y,x)^T$ is an eigenvector corresponding to the smallest eigenvalue. We therefore have that
\[P_1(x,y) = \frac{1}{\lvert (y,x)\rvert^2}\begin{pmatrix}
y\\ x
\end{pmatrix}(y,x) = \frac{1}{x^2+y^2}\begin{pmatrix}
y^2 & xy\\ xy & x^2
\end{pmatrix}\]
when $x^2+y^2\neq 0$.
Likewise,
\[P_2(x,y) = \frac{1}{\lvert (x,-y)\rvert^2}\begin{pmatrix}
x\\ -y
\end{pmatrix}(x,-y) = \frac{1}{x^2+y^2}\begin{pmatrix}
x^2 & -xy\\ -xy & y^2
\end{pmatrix}.\]
Observe that $\Hu = \lambda_1 P_1 + \lambda_2 P_2$ and that $P_1P_2 = 0$, and $P_1 + P_2 = I$ as it should. At the origin, 
$0 = \Hu = \lambda_{\alpha(1)}P_{\alpha(1)} = \frac{\lambda_1}{2}P_1 + \frac{\lambda_2}{2}P_2$ where, of course,
$\lambda_1=\lambda_2=\lambda_{\alpha(1)} = 0$ and $P_1 = P_2 = P_{\alpha(1)} = I$.
\end{example}

\section{Differentiation of the eigenvalues}

We now set out to find the conditions that makes the eigenvalues of $H\in C^1$ differentiable.
Recalling \eqref{eq:nablaPH}, and since $d_j(x)\lambda_j(x) = \tr\left(P_j(x)H(x)\right)$, the directional derivative $\DD_e\lambda_j(x) = \frac{\dd}{\dd t}\lambda_j(x+te)\rvert_{t=0}$ is formally given by
\[\DD_e\lambda_j = \frac{1}{d_j}\tr\left(P_j\DD_eH + \DD_eP_jH\right) = \frac{1}{d_j}\tr\left(P_j\DD_eH\right).\]
Although the above calculation required the assumption of a differentiable eigenprojection, we shall show that the identity still holds true whenever $P_j(x)$ has merely constant dimension $d_j$ (Proposition \ref{prop:diffoncurves}).
This yields partial derivatives, but in order to get the total derivative $\nabla\lambda_j$ it seems necessary to assume that the eigenprojection also is continuous.
Theorem \ref{thm:lambda_reg} summarizes the exact conditions for existence, and present formulas, for the gradient and the Hessian matrix of the eigenvalues.

Our main starting tools are \emph{Ky Fan's minimum principle} (Theorem 1 \cite{MR0034519}) and a powerful Lemma (Lemma 5.2 \cite{MR3155251}) restated below.

\begin{lemma}[Derivative of a minimum]\label{lem:derofmin}
Let $\M$ be a smooth compact manifold without boundary. Let $[a,b]$ be an interval of the real line and assume we have a function $U\in C^1([a,b]\times \M)$. Define the function $f$ on $[a,b]$ as
\[f(t) := \min_{p\in \M} U(t,p)\]
and let $\Xi = \Xi(t)$ be the set
\[\Xi(t) := \{m\in \M\,\vert\, U(t,m) = f(t)\}.\]
Then $f$ is Lipschitz on $[a,b]$ and the one-sided derivatives exist and are given by
\begin{align*}
\lim_{h\to 0^+}\frac{f(t+h)-f(t)}{h} &= \min_{m\in\Xi(t)}\frac{\partial}{\partial t}U(t,m),\qquad t\in[a,b),\\
\lim_{h\to 0^-}\frac{f(t+h)-f(t)}{h} &= \max_{m\in\Xi(t)}\frac{\partial}{\partial t}U(t,m),\qquad t\in(a,b].
\end{align*}	
\end{lemma}

\begin{lemma}[Ky Fan's minimum priciple]
Let $X\in S(m)$ with repeated eigenvalues $\lambda_1 \leq \cdots\leq\lambda_m$. Then for every $k = 1,\dots, m$,
\[\sum_{i=1}^k\lambda_i = \min \sum_{i=1}^k \xi_i^T X\xi_i\]
where the minimum is taken over all $k$-tuples $\{\xi_1,\dots,\xi_k\}$ of vectors in $\mR^m$ such that $\xi_i^T\xi_j = \delta_{ij}$.
\end{lemma}
As noted by Fan in the original proof, the right-hand side can be written as the minimum of $\tr\left(Q^TXQ\right)$ over all matrices $Q\in\mR^{m\times k}$ with $Q^TQ = I_k$. In light of \eqref{eq:projsplit}, it is clear that the principle may be restated in terms of projection matrices as
\[\sum_{i=1}^k\lambda_i = \min_{R\in Pr_k(m)}\tr (RX)\]
where
\[Pr_k(m) := \{P\in S(m)\,\vert\, PP = P,\,\tr P = k\},\qquad k=0,1,\dots,m,\]
are the $k$-dimensional subclasses of $Pr(m)$.
Note that $Pr_k(m)$ can be identified with the \emph{Grassmannian manifold} $Gr_k(\mR^m)$ of $k$-dimensional subspaces in $\mR^m$.

Let $\bfx\colon [a,b]\to\Omega$ be a $C^1$ curve.
By setting $\M = Pr_k(m)$ and $U\in C^1([a,b]\times \M)$ as
\[U(t, R) := \tr \left(RH(\bfx(t))\right),\]
the two lemmas imply that the sum of the $k$ smallest eigenvalues of $H\in C^1(\Omega, S(m))$ at $\bfx(t)$,
\[\ell_k(t) := \sum_{i=1}^{k}\lambda_i(\bfx(t)) = \min_{R\in Pr_k(m)}U(t,R),\]
is Lipschitz.
Furthermore, by \eqref{chainrule},
\[\frac{\partial}{\partial t}U(t,R) = \frac{\partial}{\partial t}\tr \left(RH(\bfx(t))\right)
= \tr\left(R \DD_{\bfx'(t)} H(\bfx(t))\right)\]
and the one-sided derivatives of $\ell_k(t)$ are then
\begin{equation}\label{onesidedderivatives}
\begin{aligned}
\lim_{h\to 0^+}\frac{\ell_k(t+h)-\ell_k(t)}{h} &= \min_{R\in\Xi_k(t)}\tr\left(R \DD_{\bfx'(t)} H(\bfx(t))\right),\qquad t\in[a,b).\\
\lim_{h\to 0^-}\frac{\ell_k(t+h)-\ell_k(t)}{h} &= \max_{R\in\Xi_k(t)}\tr\left(R \DD_{\bfx'(t)} H(\bfx(t))\right),\qquad t\in(a,b],
\end{aligned}
\end{equation}
where
\[\Xi_k(t) := \left\{R\in Pr_k(m)\,\vert\,\tr(R H(\bfx(t))) = \ell_k(t)\right\}.\]
We therefore want to show that $\Xi_k(t)$ is a singleton for special values of $k$. This will imply that the one-sided derivatives are equal and thus making $\ell_k$ \emph{differentiable} on $(a,b)$.

First we need a general result about projection matrices.

\begin{lemma}\label{lem:proj_stand}
	Let $P,R\in Pr(m)$. Then
	\[0\leq\tr(RP)\leq \tr P\]
	with equality on the left if and only if $RP=0$ and with equality on the right if and only if $RP = P$.
\end{lemma}

\begin{proof}
Firstly,
\[0 \leq \lVert RP\rVert^2 = \tr(RP(RP)^T) = \tr(RPR) = \tr(RP),\]
and if $\tr(RP) = 0$, then $RP=0$.
Secondly,
\begin{align*}
0 \leq \lVert RP-P\rVert^2
	&= \tr\left((RP-P)(RP-P)^T\right)\\
	&= \tr(RPR - RP - PR + P)\\
	&= \tr P - \tr(RP).
\end{align*}
Thus $\tr(RP)\leq \tr P$ and if they are equal, then $RP = P$.
\end{proof}

\begin{lemma}[The special index $j^*$]\label{singleton}
Let $X = \sum_{i=1}^{m}\lambda_iP_i/d_i \in S(m)$ and set $\ell_k := \sum_{i=1}^k\lambda_i$ to be the sum of the $k$ smallest repeated eigenvalues.
For $j=1,\dots,m$, define the indexes
\begin{equation}\label{eq:jstar_def}
j_* := \min\{i\,\vert\,\lambda_i = \lambda_j\}\qquad\text{and}\qquad j^* := \max\{i\,\vert\,\lambda_i = \lambda_j\}.
\end{equation}
Then the set $\Xi_k := \left\{R\in Pr_k(m)\,\vert\,\tr(RX) = \ell_k\right\}$ is a singleton for $k=j^*$. Namely,
\[\Xi_{j^*}  = \left\{\sum_{i=1}^{j^*}\frac{P_i}{d_i}\right\}.\]
\end{lemma}

\begin{proof}
Write $R_j := \sum_{i=1}^{j^*}P_i/d_i$. First of all, since $\tr R_j = \sum_{i=1}^{j^*}1 = j^*$, and $j^*-d_j = j_*-1=(j_*-1)^*$ (with $0^* := 0$), and
\[R_j = \sum_{i=1}^{j_*-1}\frac{P_i}{d_i} + \sum_{i=j_*}^{j^*}\frac{P_i}{d_i} = R_{j_*-1} + P_j,\]
it follows that $R_j\in Pr_{j^*}(m)$ by induction.

The leftward inclusion is clear since
\[\tr\left(R_jX\right) = \tr XR_{j_*-1} + \tr XP_j = \ell_{j_*-1} + d_j\lambda_j = \ell_{j^*}.\]
	
Now assume that $\tr(RX) = \ell_{j^*}$ for some $R\in Pr_{j^*}(m)$. We want to show that $R=R_j$. Split the matrix $Y := X - \lambda_j I$ into a negative semidefinite and positive semidefinite part as
\[Y = \sum_{i=1}^{j^*}(\lambda_i-\lambda_j)P_i/d_i + \sum_{i=j^*+1}^m(\lambda_i-\lambda_j)P_i/d_i.\]
We have $\tr(RY) = \tr(RX)-\lambda_j\tr R = \ell_{j^*} - \lambda_jj^*$, so
\begin{align*}
\ell_{j^*} - \lambda_jj^*
	&= \sum_{i=1}^{j^*}(\lambda_i-\lambda_j)\tr(RP_i)/d_i + \sum_{i=j^*+1}^m(\lambda_i-\lambda_j)\tr(RP_i)/d_i\\
	&\geq \sum_{i=1}^{j^*}(\lambda_i-\lambda_j)\tr(RP_i)/d_i + 0\\
	&\geq \sum_{i=1}^{j^*}(\lambda_i-\lambda_j)\tr P_i/d_i = \sum_{i=1}^{j^*}(\lambda_i-\lambda_j) = \ell_{j^*} - \lambda_jj^*
\end{align*}
since $0\leq\tr(RP_i)\leq \tr P_i$ by Lemma \ref{lem:proj_stand}. The inequalities are therefore equalities, which in particular means that $\sum_{i=j^*+1}^m(\lambda_i-\lambda_j)\tr(RP_i)/d_i = 0$.
Since the coefficients $\lambda_i-\lambda_j$ are positive, we must have that $\tr(RP_i) = 0$ and thus $RP_i = 0$ for $i = j^*+1,\dots,m$ by the Lemma. It follows that
\[RR_j =  R\left(I-\sum_{i=j^*+1}^mP_i/d_i\right) = R,\]
and since  $\tr(RR_j) = \tr R = j^* = \tr(R_j)$, we can conclude that also $RR_j = R_j$.
\end{proof}

\begin{proposition}\label{prop:diffoncurves}
Let $H\in C^1(\Omega,S(m))$
and write
\[ H(x) = \sum_{i=1}^m \frac{\lambda_i(x)}{d_i(x)}P_i(x).\]
Assume that the eigenprojection $P_j(x)$ to the $j$'th eigenvalue $\lambda_j(x)$ has constant dimension along a $C^1$ curve $\bfx\colon [a,b]\to\Omega$ in $\Omega$. That is,
\[\tr P_j(\bfx(t)) = d_j(\bfx(t)) = d_j = \text{const.}\]
Then $\lambda_j\circ\bfx$ is differentiable on $(a,b)$ and
\[\frac{\dd}{\dd t}\lambda_j(\bfx(t)) = \frac{1}{d_j}\tr\Big(P_j(\bfx(t))\DD_{\bfx'(t)} H(\bfx(t))\Big).\]
\end{proposition}

\begin{proof}
Let $\bfx\in C^1([a,b],\Omega)$ and let $j\in\{1,\dots,m\}$.
By the definition \eqref{eq:jstar_def}, the indexes $j_*$ and $j^*$ are in general functions of $t$. By a continuity argument (Lemma \ref{lem:semicont}) one can prove that they are constant, but it turns out that this is insignificant for the proof of the proposition. What matters is that the difference $j^* - (j_*-1) = d_j$
is constant along the curve.

Since $j_*-1 = (j_*-1)^*$ it follows by Lemma \ref{singleton} that the sets
\[\Xi_{j^*}(t) := \left\{R\in Pr_{j^*}(m)\,\vert\,\tr\Big(R H(\bfx(t))\Big) = \ell_{j^*}(t)\right\} = \left\{\sum_{i=1}^{j^*}\frac{P_i(\bfx(t))}{d_i(\bfx(t))}\right\}\]
and
\[\Xi_{j_*-1}(t) :=\left\{R\in Pr_{j_*-1}(m)\,\vert\,\tr\Big(R H(\bfx(t))\Big) = \ell_{j_*-1}(t)\right\} = \left\{\sum_{i=1}^{j_*-1}\frac{P_i(\bfx(t))}{d_i(\bfx(t))}\right\}\]
are singletons for each $t\in[a,b]$.
Here,
\[\ell_{j^*}(t) := \sum_{i=1}^{j^*}\lambda_i(\bfx(t))\qquad \text{and}\qquad \ell_{j_*-1}(t) := \sum_{i=1}^{j_*-1}\lambda_i(\bfx(t)).\]
Therefore, by the spesial case \eqref{onesidedderivatives} of Lemma \ref{lem:derofmin}, we get that the derivatives of $\ell_{j^*}$ and $\ell_{j_*-1}$ exist and that they are given by
\[\ell_{j^*}'(t) = \tr\left(\sum_{i=1}^{j^*}\frac{P_i(\bfx(t))}{d_i(\bfx(t))}\DD_{\bfx'(t)} H(\bfx(t))\right)\]
and
\[\ell_{j_*-1}'(t) = \tr\left(\sum_{i=1}^{j_*-1}\frac{P_i(\bfx(t))}{d_i(\bfx(t))}\DD_{\bfx'(t)} H(\bfx(t))\right).\]
Since
\[\ell_{j^*}(t) - \ell_{j_*-1}(t) = \sum_{i=j_*}^{j^*}\lambda_i(\bfx(t)) = d_j\lambda_j(\bfx(t)),\]
and
\[\sum_{i=1}^{j^*}\frac{P_i(\bfx(t))}{d_i(\bfx(t))} - \sum_{i=1}^{j_*-1}\frac{P_i(\bfx(t))}{d_i(\bfx(t))} = \sum_{i=j_*}^{j^*}\frac{P_i(\bfx(t))}{d_i(\bfx(t))} = P_j(\bfx(t)),\]
it follows that
\[d_j\frac{\dd}{\dd t}\lambda_j(\bfx(t)) = \ell_{j^*}'(t) - \ell_{j_*-1}'(t) = \tr\left(P_j(\bfx(t))\DD_{\bfx'(t)} H(\bfx(t))\right)\]
which is what we wanted to prove.
\end{proof}

This completes the proof of Theorem \ref{thm:projection_directional_derivative}, and the next proposition completes the proof of Theorem \ref{thm:projectionderivative} by showing that the two assumptions \eqref{eq:projectionderivative_condition} are equivalent.

\begin{proposition}\label{prop:equivalentprop}
Let $H\in C^1(\Omega,S(m))$
and write
\[ H(x) = \sum_{i=1}^m \frac{\lambda_i(x)}{d_i(x)}P_i(x).\]
The following are equivalent for $j\in\{1,\dots,m\}$ in $\Omega$.
\begin{enumerate}[(a)]
\item $\lambda_j$ is $C^1$ and $P_j$ has constant dimension.
\item $\lambda_j$ is differentiable and $P_j$ has constant dimension.
\item $P_j$ is differentiable.
\item $P_j$ is continuous.
\end{enumerate}
\end{proposition}

\begin{proof}
(a) $\Rightarrow$ (b) is immediate and the proof of Theorem \ref{thm:projectionderivative} gives (b) $\Rightarrow$ (c). The step (c) to (d) is again trivial, and if assuming (d), the dimension of $P_j$ is of course constant and by Proposition \ref{prop:diffoncurves}, the partial derivatives of $\lambda_j$ exists and are given by
\[\frac{\partial}{\partial x_i}\lambda_j(x) = \DD_{e_i}\lambda_j(x) = \frac{1}{d_j}\tr\Big(P_j(x)\DD_{e_i} H(x)\Big).\]
Thus (d) $\Rightarrow$ (a) since the partial derivatives then are seen to be continuous and $\lambda_j$ is therefore differentiable with a continuous gradient.
\end{proof}

We now gather the various regularity properties for eigenvalues of symmetric matrices. For completeness, we also record some statements valid when $ H$ is only continuous.

\begin{theorem}\label{thm:lambda_reg}
Let $H\colon\Omega\to S(m)$ be given 
and write
\[ H(x) = \sum_{i=1}^m \frac{\lambda_i(x)}{d_i(x)}P_i(x)\]
where $\lambda_1(x)\leq\cdots\leq\lambda_m(x)$ and $d_i(x) = \tr P_i(x)$.
The following hold for $j\in\{1,\dots,m\}$.
\begin{enumerate}[(A)]
\item Zero'th order properties. Assume that $H\in C(\Omega,S(m))$.
\begin{enumerate}[(i)]
\item $\lambda_j$ is continuous in $\Omega$.
\item If $H\in C^{\alpha}(\Omega,S(m))$ for some $0<\alpha\leq 1$, then $\lambda_j\in C^{\alpha}(\Omega)$ with the same Hölder/Lipchitz-constant as $H$.
\end{enumerate}
\item First order properties. Assume that $H\in C^1(\Omega,S(m))$.
\begin{enumerate}[(i)]
\item If $P_j$ has constant dimension $\tr P_j(x) = d_j$ in $\Omega$, then $\lambda_j$ has directional derivatives $\DD_e\lambda_j$ satisfying
\[\DD_e\lambda_j(x)\cdot P_j(x) = P_j(x)\DD_{e} H(x)P_j(x)\]
in every direction $e\in\mR^n$.
In particular,
\[\DD_e\lambda_j(x) = \frac{1}{d_j}\tr\Big(P_j(x)\DD_{e} H(x)\Big)\]
and
\[\DD_e\lambda_j(x) = \xi^T\DD_e H(x)\xi = \xi^T\nabla_\xi H(x) e\]
for any $\xi\in P_j(x)(\mR^m)\cap\mathbb{S}^{m-1}$.
\item Suppose that one (and therefore all) of the conditions (a)-(d) from Proposition \ref{prop:equivalentprop} holds. Then
\[\vert\nabla\lambda_j(x)\vert \leq \frac{1}{\sqrt{d_j}}\max_{e\in\mathbb{S}^{n-1}}\lVert\DD_e H(x)\rVert,\]
and
\begin{equation}\label{3.}
\nabla\lambda_j(x) = \xi^T \nabla_{\xi} H(x)
\end{equation}
for any $\xi\in P_j(x)(\mR^m)\cap\mathbb{S}^{m-1}$. Alternatively,
\begin{equation}\label{gradformula}
\nabla\lambda_j(x)e = \DD_e\lambda_j(x) = \frac{1}{d_j}\tr\Big(P_j(x)\DD_{e} H(x)\Big)
\end{equation}
for every $e\in\mR^n$.
\end{enumerate}
\item Second order properties. Assume that $H\in C^2(\Omega,S(m))$.
\begin{enumerate}[(i)]
\item Let $a\in\mR^n$. If $P_j$ has constant dimension $\tr P_j(x) = d_j$ in $\Omega$, then $\DD_a\lambda_j$ exists and has directional derivatives $\DD_b\DD_a\lambda_j$ satisfying
\begin{align*}
\DD_b\DD_a\lambda_j(x)\cdot P_j(x)
	&= P_j(x) \Big(\DD_{b}\DD_a H(x)\\
	&\qquad{}+ \DD_a H(x)A_j(x)\DD_{b}H(x)\\
	&\qquad{}+ \DD_b H(x)A_j(x)\DD_{a}H(x)\Big) P_j(x)
\end{align*}
in every direction $b\in\mR^n$. In particular,
\[\DD_b\DD_a\lambda_j(x) = \frac{1}{d_j}\tr\Big(P_j(x) \big[\DD_{b}\DD_a H(x) + 2\DD_a H(x)A_j(x)\DD_{b}H(x)\big]\Big)\]
and
\begin{align*}
\DD_b\DD_a\lambda_j(x)
	&= \xi^T\Big(\DD_{b}\DD_a H(x) + 2\DD_a H(x)A_j(x)\DD_{b}H(x)\Big)\xi\\
	&= a^T\Big(\nabla_{\xi}(\nabla_{\xi} H)^T(x) + 2(\nabla_{\xi} H(x))^T A_j(x)\nabla_{\xi} H(x)\Big)b
\end{align*}
for any $\xi\in P_j(x)(\mR^m)\cap\mathbb{S}^{m-1}$.
\item Suppose that one (and therefore all) of the conditions (a)-(d) from Proposition \ref{prop:equivalentprop} holds. Then $\nabla\lambda_j$ is differentiable in $\Omega$
with Hessian $\mathcal{H}\lambda_j := \nabla(\nabla\lambda_j^T)$ given by
\begin{equation}\label{eq:H3}
\mathcal{H}\lambda_j(x) = \nabla_{\xi}(\nabla_{\xi} H)^T(x) + 2(\nabla_{\xi} H(x))^T A_j(x)\nabla_{\xi} H(x)
\end{equation}
for any $\xi\in P_j(x)(\mR^m)\cap\mathbb{S}^{m-1}$. Alternatively,
\begin{equation}\label{eq:lambdaHessiantensor}
\begin{aligned}
a^T\mathcal{H}\lambda_j(x)b
	&= \DD_b\DD_a\lambda_j(x)\\
	&= \frac{1}{d_j}\tr\Big(P_j(x) \big[\DD_{b}\DD_a H(x) + 2\DD_a H(x)A_j(x)\DD_{b}H(x)\big]\Big)
\end{aligned}
\end{equation}
for every $a,b\in\mR^n$.
\end{enumerate}
\end{enumerate}
\end{theorem}

\begin{remark}
\begin{itemize}
\item Note that the eigenvectors $\xi$ depend on $x\in\Omega$.
\item In (C), the matrices $\nabla_{p}(\nabla_{q} H)^T$ and $\DD_{b}\DD_a H$ represent the second order derivatives of $H$. The first one is the Hessian matrix of the $C^2$ function $x\mapsto p^T H(x)q$, and is therefore both in $S(n)$ and symmetric in $p$ and $q$. The latter is the appropriate linear combination of the second order partial derivatives $\DD_{e_i}\DD_{e_k}H(x) = \frac{\partial^2}{\partial x_i\partial x_k}H(x)\in S(m)$.
\item Unlike in the case of the gradient, the expressions for $\mathcal{H}\lambda_j$ contain the matrix $A_j$ that cannot be assumed to be continuous or bounded.
\end{itemize}
\end{remark}

\begin{proof}[Proof of (A)]
It is a well known fact that eigenvalues depend continuously on the matrix. Corollary 6.3.8 in \cite{MR2978290} gives the estimate
\[\sum_{i=1}^m\lvert \lambda_i-\tilde{\lambda_i}\rvert^2 \leq \lVert E\rVert^2\]
whenever $\lambda_1\leq\cdots\leq\lambda_m$ are the eigenvalues of $H_0\in S(m)$, and $\tilde{\lambda}_1\leq\cdots\leq\tilde{\lambda}_m$ are the eigenvalues of $H_0 + E\in S(m)$.
Our results follow by setting $H_0 = H(x)$, $E = H(x) - H(x+y)$, and by using that $\lVert E\rVert = o(1)$ and $\lVert E\rVert \leq C\lvert y\rvert^\alpha$, respectively.
\end{proof}

\begin{proof}[Proof of (B)]
Part (i): Proposition \ref{prop:diffoncurves} yields the formula $\DD_e\lambda_j = \frac{1}{d_j}\tr(P_j\DD_e H)$. But since also $P_j$ is directionally differentiable by Theorem \ref{thm:projection_directional_derivative}, the derivative of $\lambda_j P_j = HP_j$ is $\DD_e\lambda_j\cdot P_j + \lambda_j\DD_e P_j = \DD_e H P_j + H\DD_e P_j$ and the more general formula
\[\DD_e\lambda_j\cdot P_j = P_j\DD_e H P_j\]
is obtained by multiplying on the left with $P_j$.

For Part (ii), assume that the conditions (a)-(d) from Proposition \ref{prop:equivalentprop} hold.

Write the length of the gradient as $\max_{e\in\mathbb{S}^{n-1}}\nabla\lambda_j(x)e$ and use \eqref{gradformula} to get
\begin{align*}
\vert\nabla\lambda_j(x)\vert
	&= \max_{e\in\mathbb{S}^{n-1}}\frac{1}{d_j}\tr\Big(P_j(x)\DD_e H(x)\Big)\\
	&\leq \max_{e\in\mathbb{S}^{n-1}}\frac{1}{d_j} \lVert P_j(x)\rVert\lVert\DD_e H(x)\rVert\\
	&= \frac{1}{\sqrt{d_j}}\max_{e\in\mathbb{S}^{n-1}}\lVert\DD_e H(x)\rVert.
\end{align*}

\end{proof}

\begin{proof}[Proof of (C)]
Part (i): Theorem \ref{thm:projection_directional_derivative} and (B) (i) implies that $\DD_a\lambda_j = \frac{1}{d_j}\tr(P_j\DD_aH)$ and that it is directionally differentiable when $H$ is $C^2$. The identity $\DD_a\lambda_j\cdot P_j = P_j\DD_a H P_j$ can therefore be differentiated yielding
\[\DD_b\DD_a\lambda_j\cdot P_j + \DD_a\lambda_j\cdot \DD_b P_j = \DD_{b}P_j\DD_a H P_j + P_j \DD_{b}\DD_a H P_j + P_j \DD_a H\DD_b P_j.\]
Multiply from both sides with $P_j$ and use that $P_j\DD_bP_j = P_j\DD_b H A_j$ and $P_j\DD_b P_j P_j = 0$ to get
\begin{align*}
\DD_b\DD_a\lambda\cdot P_j
	&= P_j\DD_b H A_j\DD_a H P_j + P_j \DD_{b}\DD_a H P_j + P_j \DD_a HA_j\DD_b H P_j\\
	&= P_j\Big(\DD_{b}\DD_a H + \DD_b H A_j\DD_a H + \DD_a HA_j\DD_b H\Big) P_j.
\end{align*}
The other identities follow from the cyclic property of the trace and the symmetry of the factors, and by using \eqref{sym1} and \eqref{sym2}.

Part (ii): When the conditions (a)-(d) from Proposition \ref{prop:equivalentprop} hold, formula \eqref{gradformula} shows that the gradient $\nabla\lambda$ is differentiable when $H$ is $C^2$. The rest follows from (i).
\end{proof}

Analogous formulas for \emph{one-sided} directional derivatives are given in \cite{MR1868352}. There the derivatives of $\lambda_j$ are expressed in terms of a specific eigenvalue of certain matrices. For example, and in our notation, the first order derivative is given as a particular eigenvalue of the $d_j\times d_j$ symmetric matrix $Q^T\DD_e H Q$ where $Q = (\xi_1,\dots,\xi_{d_j})\in\mR^{m\times d_j}$ is an eigenvector matrix corresponding to $\lambda_j$. This interpretation is valid also for the formulas presented in Theorem \ref{thm:lambda_reg} since $P_j = QQ^T$, and by (B) part (i),
\[Q^T\DD_e H Q = Q^TP_j\DD_e H P_jQ = \DD_e\lambda_j\cdot Q^TP_jQ = \DD_e\lambda_j\cdot Q^TQ = \DD_e\lambda_j\cdot I_{d_j}\]
and $Q^T\DD_e H Q$ is just a scaling of the identity matrix.

\section{Asymptotic expansion of the eigenvalues}

We conclude the paper by inserting the various expressions for the first- and second order derivatives of the eigenvalues into the expansions
\begin{align*}
\lambda_j(x+he) &= \lambda_j(x) + h\DD_e\lambda_j(x) + \frac{1}{2}h^2\DD_e\DD_e\lambda_j(x) + o(h^2),\\
\lambda_j(x+y) &= \lambda_j(x) + \nabla\lambda_j(x)y + \frac{1}{2}y^T\mathcal{H}\lambda_j(x)y + o(\lvert y\rvert^2).
\end{align*}
Recall that if $x\mapsto H(x)\in S(m)$ has repeated eigenvalues and eigenprojections $\lambda_i(x)\in\mR$ and $P_i(x)\in Pr(m)$, $i = 1,\dots, m$, then the pseudoinverse of $\lambda_j(x) I-H(x)$ is given by
\[A_j(x) := \sum_{ \substack{ i=1\\ \lambda_i(x) \neq \lambda_j(x)} }^m \frac{P_i(x)/d_i(x)}{\lambda_j(x) - \lambda_i(x)},\qquad d_i(x) = \tr P_i(x).\]

\begin{corollary}[Second order directional expansion]
Let $H\colon\Omega\to S(m)$ be $C^2$ in a domain $\Omega\subseteq\mR^n$ and let $j\in\{1,\dots,m\}$.
If the eigenprojection $P_j$ has constant dimension $d_j=\tr P_j$, then for every $x\in\Omega$ and for every direction $e\in\mR^n$,
\begin{align*}
\lambda_j(x+he)
&= \xi^T\Big(\lambda_j I + h\DD_eH + \tfrac{1}{2}h^2\DD_e\DD_eH + h^2\DD_eHA_j\DD_eH\Big)\xi + o(h^2)\\
&= \lambda_j + h\xi^T\nabla_\xi H e + \frac{1}{2}h^2e^T\Big(\nabla_\xi(\nabla_\xi H)^T + 2(\nabla_\xi H)^TA_j\nabla_\xi H\Big)e + o(h^2)
\end{align*}
as $h\to 0$ for any $\xi\in P_j(x)(\mR^m)\cap\mathbb{S}^{m-1}$. Alternatively,
\[\lambda_j(x+he) = \frac{1}{d_j}\tr\left(P_j\big[\lambda_j I + h\DD_eH + \tfrac{1}{2}h^2\DD_e\DD_eH + h^2\DD_eHA_j\DD_eH\big]\right)  + o(h^2)\]
\end{corollary}

The functions on the right-hand sides are all evaluated at $x$. Of course, the corresponding first order expressions are also valid if $H$ is $C^1$.

The assumption of constant dimension of $P_j$ is sufficient for directional expansion. But in order to get the total asymptotic behavior we also need to assume that $\lambda_j$ is differentiable or, equivalently, that $P_j$ is continuous.
We want to add one more condition to this list.

\begin{proposition}\label{prop:noncrossing}
	Assume that $H\colon\Omega\to S(m)$ is $C^k$ in $\Omega$ for some $k = 0,1,2,\dots$. If the number of distinct eigenvalues of $H$ is constant, then every eigenvalue and eigenprojection of $H$ is $C^k$ in $\Omega$. 
\end{proposition}

\begin{lemma}[Semicontinuity of some integer-valued functions associated to symmetric matrices]\label{lem:semicont}
Let $H\in C(\Omega,S(m))$.
\[ H(x) = \sum_{i=1}^m \frac{\lambda_i(x)}{d_i(x)}P_i(x),\qquad d_i(x) = \tr P_i(x).\]
For $j=1,\dots,m$, define the indexes
$j_*(x) := \min\{i\,\vert\,\lambda_i(x) = \lambda_j(x)\}$ and $j^*(x) := \max\{i\,\vert\,\lambda_i(x) = \lambda_j(x)\}$,
and let $s_j(x) := \lvert\{\lambda_1(x),\dots,\lambda_j(x)\}\rvert$ be the number of distinct eigenvalues less or equal to $\lambda_j(x)$. Then $j_*$ and $s_j$ are lower semicontinuous, and $j^*$ and $d_j$ are upper semicontinuous in $\Omega$. If $d_j$ is constant, then $j_*$ and $j^*$ are also constant. Furthermore, the number $s_m(x)$ of distinct eigenvalues of $ H$ at $x$ satisfies $s_m(x) = \sum_{i=1}^m 1/d_i(x)$, and if $s_m$ is constant, then so is every $d_j$.
\end{lemma}

\begin{proof}
We see that $j_*$ and $s_j$ decrease only if two different eigenvalues become equal. Since the eigenvalues of $ H$ are continuous (Theorem \ref{thm:lambda_reg} (A)), the superlevelsets $\{x\,\vert\, j_*(x)>c\}$ and $\{x\,\vert\, s_j(x)>c\}$ are open and the functions are therefore l.s.c.
The same reasoning applies when arguing that $j^*$ and $d_j$ are u.s.c.
There is however a more instructive proof of the upper semicontinuity of $d_j$. Since
\[0 = ( H(x) - \lambda_j(x)I)P_j(x) = ( H(x_0) - \lambda_j(x_0)I)P_j(x) + o(1)\]
as \(x\to x_0\), multiplying on the left by \(A_j(x_0)\) and rearranging gives $P_j(x) = P_j(x_0)P_j(x) + o(1)$.
Lemma \ref{lem:proj_stand} then implies that $d_j(x)\leq d_j(x_0) + o(1)$ and it follows that $\limsup_{x\to x_0}d_j(x) \leq d_j(x_0)$.
	
If $d_j$ is constant, then $j^* = j_* + d_j - 1$ is also l.s.c. It is therefore continuous and thus constant, which in turn makes $j_*$ constant.
	
Evaluate $s:=s_m = \lvert\{\lambda_1,\dots,\lambda_m\}\rvert$ at some fixed $x\in\Omega$. As in Section \ref{sec:prelim}, we choose a re-indexing $\alpha\colon\{1,\dots,s\}\to\{1,\dots,n\}$ so that $l\mapsto \lambda_{\alpha(l)}$ is a bijection.
Since
\[\sum_{\substack{ i=1\\ \lambda_i = \lambda_{\alpha(l)}} }^m\frac{1}{d_i} = \frac{1}{d_{\alpha(l)}}\sum_{\substack{ i=1\\ \lambda_i = \lambda_{\alpha(l)}} }^m 1 = \frac{1}{d_{\alpha(l)}}d_{\alpha(l)} = 1,\]
we get that
\[s = \sum_{l=1}^s 1 = \sum_{l=1}^s\sum_{\substack{ i=1\\ \lambda_i = \lambda_{\alpha(l)}} }^m\frac{1}{d_i} = \sum_{i=1}^m\frac{1}{d_i}.\]
	
Finally, since
each \(d_i\) is u.s.c., \(\frac{1}{d_i}\) is l.s.c., and \(-\frac{1}{d_i}\) is u.s.c. So if $s$ is constant, then
\[\frac{1}{d_j(x)} = s - \sum_{ \substack{i=1\\ i\neq j} }^m\frac{1}{d_i(x)}\]
is u.s.c. Thus \(d_j\) is also l.s.c. and therefore continuous and constant.
\end{proof}

\begin{proof}[Proof of Proposition \ref{prop:noncrossing}]
By the Lemma, every eigenprojection has constant dimension and we can therefore re-index the eigenvalues and eigenprojections independently of $x\in\Omega$. See \eqref{eq:X_unrep}. After renaming we can write
\[H(x) = \sum_{i=1}^s\lambda_i(x)P_i(x)\]
where
\[\lambda_1(x)<\cdots<\lambda_s(x),\qquad P_i(x)P_j(x) = \delta_{ij}P_i(x),\qquad \sum_{i=1}^s P_i(x) = I,\]
for all $x\in\Omega$.
Moreover, the Frobenius formula \eqref{eq:frobenius} becomes
\begin{equation}\label{eq:frobeniusII}
P_j(x) = \prod_{\substack{i=1\\ i\neq j}}^s\frac{H(x) - \lambda_i(x)I}{\lambda_j(x) - \lambda_i(x)}
\end{equation}
and shows that $P_j$ has (at least) the same regularity as $\oll := (\lambda_1,\dots,\lambda_s)^T$.
In particular, each $P_j$ and $\lambda_j$ is continuous whenever $H$ is continuous by Theorem \ref{thm:lambda_reg} (A).

Assume next that $H$ is $C^k$ for some $k\geq 1$. By the above, $P_j$ is continuous and $\lambda_j$ is $C^1$ (and therefore also $P_j$) by Theorem \ref{thm:lambda_reg} (B).

Since $\nabla\lambda_j(x)e = \frac{1}{d_j}\tr\left(P_j(x)\DD_eH(x)\right)$, the derivative of every $\lambda_j$ is a smooth function of $P_j$ -- which again is a smooth function of the eigenvalues and $H$ -- and the tensor $\DD H$.
In symbols,
\[\nabla\oll = F(\oll,H,\DD H),\]
and the result follows by induction.
\end{proof}

\begin{corollary}[Second order total expansion]
Let $H\colon\Omega\to S(m)$ be $C^2$ in a domain $\Omega\subseteq\mR^n$ and let $j\in\{1,\dots,m\}$.
Assume that \textbf{one} of the following conditions hold in $\Omega$.
\begin{enumerate}[(1)]
\item The eigenprojection $P_j$ is continuous.
\item $\lambda_j$ is differentiable and $d_j = \tr P_j$ is constant.
\item The number of distinct eigenvalues of $H$ is constant.
\end{enumerate}
Then, for every $x\in\Omega$,
\begin{align*}
\lambda_j(x+y)
&= \xi^T\Big(\lambda_j I + \DD_yH + \tfrac{1}{2}\DD_y\DD_yH + \DD_yHA_j\DD_yH\Big)\xi + o(\lvert y\rvert^2)\\
&= \lambda_j + \xi^T\nabla_\xi H y + \frac{1}{2}y^T\Big(\nabla_\xi(\nabla_\xi H)^T + 2(\nabla_\xi H)^TA_j\nabla_\xi H\Big)y + o(\lvert y\rvert^2)
\end{align*}
as $y\to 0$ for any $\xi\in P_j(x)(\mR^m)\cap\mathbb{S}^{m-1}$. Alternatively,
\[\lambda_j(x+y) = \frac{1}{d_j}\tr\left(P_j\big[\lambda_j I + \DD_yH + \tfrac{1}{2}\DD_y\DD_yH + \DD_yHA_j\DD_yH\big]\right)  + o(\lvert y\rvert^2)\]
\end{corollary}

Observe that if $m=2$, then constant dimension of an eigenprojection implies (3). Also when $n=1$, the expansions can be written as
\begin{align*}
\lambda_j(t+h)
	&= \frac{1}{d_j}\tr\left(P_j\big[\lambda_j I + hH' + \tfrac{1}{2}h^2H'' + h^2H'A_jH'\big]\right)  + o(h^2)\\
	&= \lambda_j + h\xi^TH'\xi + \frac{1}{2}h^2\xi^T\Big(H'' + 2H'A_jH'\Big)\xi  + o(t^2)
\end{align*}
and it is again enough to assume a constant $d_j$ since the directional and total derivatives are equivalent on the real line.

\paragraph{Acknowledgements:}
Supported by the Academy of Finland (grant SA13316965), and Aalto University. I thank Juha Kinnunen, Peter Lindqvist, and Fredrik Arbo H\o eg.

\bibliographystyle{alpha}
\bibliography{C:/LocalUserData/User-data/brustak1/references}

\textsc{Karl K. Brustad
	\hfill\break\indent
	Department of Mathematics and System Analysis
	\hfill\break\indent
	Aalto University
	\hfill\break\indent FI-00076, Aalto, Finland
	\hfill\break\indent
	{\tt karl.brustad@aalto.fi}
}

\end{document}